\documentclass[english]{article}
\usepackage{amsfonts}
\usepackage[latin9]{inputenc}
\usepackage{amsmath}
\usepackage{amssymb}
\usepackage{graphicx}
\usepackage[pdfstartview=FitH]{hyperref}
\usepackage{amsthm}
\usepackage{authblk}
\usepackage[all]{xy}


\usepackage{babel}
\begin{document}
\title{Two approaches towards property (T) for groups acting on simplicial complexes}
\author{Izhar Oppenheim}
\affil{Department of Mathematics, The Technion-Israel Institute of Technology, 32000 Haifa, Israel \\ 
E-mail: izharo@gmail.com}

\newtheorem{theorem}{Theorem}[section]
\newtheorem{lemma}[theorem]{Lemma}
\newtheorem{definition}[theorem]{Definition}
\newtheorem{claim}[theorem]{Claim}
\newtheorem{example}[theorem]{Example}
\newtheorem{remark}[theorem]{Remark}
\newtheorem{proposition}[theorem]{Proposition}
\newtheorem{corollary}[theorem]{Corollary}
\maketitle
\textbf{Abstract}. This article generalizes two approaches for property (T) - the first is a generalization of \.{Z}uk's criterion for property (T) and the second is a generalization of the work of Kassabov regarding property (T) and subspace arrangements.  In both cases we obtain new criteria for property (T) (and for vanishing of higher $L^2$-cohomologies).   
\section{Introduction}
We'll start by briefly discussing two approaches toward proving property (T).

\subsection{Criteria for property (T) regarding Laplacian eigenvalues} 
For a finite graph $L$ with a set of vertices $V_{L}$, the Laplacian
of the graph $\Delta^+$ is an operator on the space of real valued
functions on $V_{L}$ which is defined as \[
\Delta^+ f(v)=f(v)-\frac{1}{m(v)}\sum_{u\sim v}f(u)\]
where $m(v)$ is the valance of $v$ and $u\sim v$ means that there
is an edge connecting $u$ and $v$. The Laplacian is a positive operator and we denote by $\lambda(L)$ its smallest positive eigenvalue. One can generalize the definition of the Laplacian
so it will be defined for a simplicial complex  $X$ of any dimension. For
such a complex the Laplacian is again a positive
operator and we denote by $\lambda(X)$ its smallest positive eigenvalue.

Ballmann and \'{S}wiatkowski in \cite{BS} and \.{Z}uk in \cite{Zuk} gave criteria
for the vanishing of the $L^{2}$-cohomology of a group $\Gamma$
acting on a simplicial complex $X$, by considering the values of
$\lambda$ for the links of $X$. More in specifically \cite{BS} and
\cite{Zuk} proved the following theorems:

\newtheorem*{BS}{\cite[Theorem 2.5]{BS}}

\begin{BS}
\label{thmbs}

Let $X$ be a locally finite simplicial complex of dimension $n$
and let $\Gamma$ be a group acting properly discontinuously and by
automorphisms on $X$. Assume that for every simplex $\eta$ of dimension
$k-1$, the link of $\eta$ denoted by $X_{\tau}$ is a connected
simplicial complex and that there is $\varepsilon>0$ such that $\lambda(X_{\eta})\geq\frac{k(n-k)}{k+1}+\varepsilon$, then $L^{2}H^{k}(X,\rho)=0$ for any unitary representation $\rho$ of $\Gamma$.

\end{BS}

\newtheorem*{Z}{\cite[Theorem 1]{Zuk}}

\begin{Z}

Let $X$ be a locally finite simplicial complex of dimension $2$
and let $\Gamma$ be a group acting properly discontinuously and freely
by automorphisms on $X$. Assume that for every vertex $v$, the link
of $v$ denoted by $X_{v}$ is a connected graph and that for every
two connected vertices $u,v$ in $X$ we have the following inequality
$\frac{\lambda(X_{u})+\lambda(X_{v})}{2}>\frac{1}{2}$, then $L^{2}H^{1}(X,\rho)=0$
for any unitary representation $\rho$ of $\Gamma$.

\end{Z}

In \cite{DJ1} the above theorems were generalized by Dymara and Januszkiewicz to a more general setting in which $\Gamma$ isn't necessarily discrete but just locally compact and unimodular.

If $\Gamma$ acts cocompactly and properly 
discontinuously by automorphisms on a contractible simplicial complex $X$, then the following are equivalent:
\begin{enumerate}
\item $\Gamma$ has property (T).
\item $L^{2}H^{1}(X,\rho)=0$ for any unitary representation $\rho$ of
$\Gamma$ (see \cite{BHV}).
\item $L^{2}\mathcal{H}^{1}(X,\rho)=0$ for any unitary representation $\rho$ of
$\Gamma$, where $L^{2}\mathcal{H}^{1}(X,\rho)$ is the first reduced $L^2$-cohomology (see \cite{BHV}, \cite{Shalom}). 
\end{enumerate} 
 
Therefore the above theorems give criteria for property (T) when $\Gamma$ acts cocompactly and properly
on a contractible locally finite simplicial complex of dimension $2$. 

\subsection{Criterion for property (T) regarding angles}
A different approach towards property (T) was taken in by Kassabov in \cite{K}. In \cite{K} a concept of a cosine, denote as $cos(V,U)$, between two Hilbert subspaces $V,U$ is defined. This concept is used to in the following theorem:
\newtheorem*{K}{\cite[Theorem 5.1]{K}}
\begin{K}
Let  $V_1,...,V_n$ be closed subspaces of a Hilbert space $H$. Suppose that the $n \times n$ symmetric matrix 
$$ A = \left( \begin{array}{ccccc}
 1 & -\varepsilon_{12} &  -\varepsilon_{13} & \cdots & -\varepsilon_{1n} \\
 -\varepsilon_{21}  & 1 &  -\varepsilon_{23} & \cdots & -\varepsilon_{2n} \\
-\varepsilon_{31}  & -\varepsilon_{32} &  1 & \cdots & -\varepsilon_{3n} \\
\vdots  & \vdots &  \vdots & \ddots & \vdots \\
-\varepsilon_{n1}  & -\varepsilon_{n2} &  -\varepsilon_{n3} & \cdots & 1 

\end{array} \right)$$
where $\varepsilon_{ij}=cos(V_i,V_j)$ is positive definite. Then for every $v \in H$ we have 
$$d (v,\bigcap V_i) \leq d_v^t A^{-1} d_v$$
where $d_v$ is the column vector with entries $d(v,V_i)$.  
\end{K}

This theorem can be used to give a criterion for property (T) in the case that $\Gamma = \langle G_1,...,G_n \rangle$, where $G_1,...,G_n$ are compact groups. 
The criterion is as follows: for a unitary representation $\rho: \Gamma \rightarrow \mathcal{U} (H_\rho)$ denote by $H_\rho^{G_i}$ the closed subspace of $H_\rho$ fixed by $G_i$. Define 
$$\varepsilon_{ij} = inf_\rho \lbrace cos(H_\rho^{G_i}, H_\rho^{G_j}) \rbrace$$ 
if 
$$ A = \left( \begin{array}{ccccc}
 1 & -\varepsilon_{12} &  -\varepsilon_{13} & \cdots & -\varepsilon_{1n} \\
 -\varepsilon_{21}  & 1 &  -\varepsilon_{23} & \cdots & -\varepsilon_{2n} \\
-\varepsilon_{31}  & -\varepsilon_{32} &  1 & \cdots & -\varepsilon_{3n} \\
\vdots  & \vdots &  \vdots & \ddots & \vdots \\
-\varepsilon_{n1}  & -\varepsilon_{n2} &  -\varepsilon_{n3} & \cdots & 1 

\end{array} \right)$$
is a positive matrix, then $\Gamma$ has property (T) (we refer the reader to \cite[Observation 2.1]{K} for the explanation of the connection between the theorem and the criterion).

Observe that this theorem can be applied in the case that $\Gamma$ is a group acting properly on an $n$-dimensional connected simplicial complex and the fundamental domain is a single ($n$-dimensional) simplex denoted by $\sigma$. In this case, $\Gamma$ is generated by the stabilizers of the $(n-1)$-dimensional faces of $\sigma$ denoted by $G_1,...,G_{n+1}$. Note that since the action is proper, we get that $G_i$ is compact for every $i$.  

\subsection{New criteria for property (T)}

In this article we shall generalize the above approaches to get new criteria for vanishing of $L^2$-cohomologies for groups acting on simplicial complexes. 
The two generalizations are of different nature: the generalization of the \cite[Theorem 2.5]{BS} and \cite[Theorem 1]{Zuk} are basically some sort of averaging along a simplex of the first positive eigenvalues of the Laplacian at the links. For the $2$-dimensional case we get the following theorem: 
\newtheorem*{Theorem 1}{Theorem 1}
\begin{Theorem 1}
Let $\Gamma$ be a locally compact, properly discontinuous, unimodular group of automorphisms
of $X$ acting cocompactly on $X$ such that $X$ is a locally finite  contractible $2$-dimensional simplicial complex. For a $1$-simplex $(u,v)$ of $X$ denote
$$S_{(u,v)}= \lambda (X_u) + \lambda (X_v) -1$$
If for every $2$-simplex $(u,v,w)$ of $X$ the following holds:
\begin{itemize}
\item $ \lambda (X_u) + \lambda (X_v) + \lambda (X_w) >\frac{3}{2}$
\item $ S_{(u,v)}S_{(u,w)} + S_{(u,v)}S_{(v,w)} + S_{(u,w)}S_{(v,w)} >0$
\end{itemize}
then $\Gamma$ has property (T).
\end{Theorem 1}

There are some interesting examples for the above theorem coming from the sporadic finite simple groups. Here is one such example: in \cite{Kantor} it is shown that the Lyons group (which is a sporadic finite simple group) acts on a finite simplicial complex $Y$ and that the universal cover of $Y$ is an exotic affine building. This yields an extension of the Lyons group by $\pi_1 (Y)$: 
$$ 		0 \longrightarrow \pi_1 (Y) \longrightarrow \widetilde{Lyons} \longrightarrow Lyons \longrightarrow 0 $$
As it turns out, \.{Z}uk's criterion doesn't show that $\pi_1 (Y)$ has property (T) but our criterion holds and we are able to prove property (T) for $\pi_1 (Y)$ and $\widetilde{Lyons}$. There are $2$ other examples of this nature in the last section of this article.  
In addition, there are some hyperbolic Kac-Moody groups where our criterion holds and \.{Z}uk's criterion fails. \\ \\

The generalization of \cite[Theorem 5.1]{K} deals with the case that the fundamental domain of the group action is not a single simplex (as in \cite{K}) but a finite simplicial complex. In order to state our theorem  we introduce the following terminology: let $Y=(V,E)$ be a connected finite graph (with no loops or multiple edges), $\Lambda$ be a subgroup of $Aut(Y)$ and $\rho: \Lambda \rightarrow \mathcal{U} (H)$ be a unitary representation ($H$ is a Hilbert space). Denote
 $$C^0 (Y,\rho)= \lbrace \phi : V \rightarrow H : \phi \text{ is equivariant} \rbrace$$
Define the reduced cosine of $Y$ with respect to $\Gamma$ and $\rho$ to be 
$$ cos_r (Y,\Lambda,\rho) = sup \lbrace \dfrac{\sum_{(u,v) \in E}  \langle \phi (u) , \phi (v) \rangle }{\sum_{(u,v) \in E} \vert \phi (u)\vert \vert \phi (v) \vert}: 0 \neq \phi \in C^0 (Y,\rho), \sum_{(u,v) \in E} \phi (u) + \phi (v) =0  \rbrace$$ 
and define
$$ cos_r (Y,\Lambda) = sup \lbrace cos_r (Y,\Gamma,\rho) : \rho \text{ is a unitary representation of } \Gamma \rbrace$$
Now we can state the theorem for in the $2$-dimensional case:  

\newtheorem*{Theorem 2}{Theorem 2}
\begin{Theorem 2}
Let $\Gamma$ be a locally compact, properly discontinuous, unimodular group of automorphisms
of $X$ acting cocompactly on $X$ such that $X$ is a locally finite  contractible $2$-dimensional simplicial complex. For a vertex $u$ denote by $\Gamma_u$ the stabilizer of $u$ in $\Gamma$ acting on $X_u$. 
If for every $2$-simplex $(u,v,w)$ of $X$ the matrix:
$$ A (u,v,w) = \left( \begin{array}{ccc}
 1 & -cos_r (X_u, \Gamma_u) &  -cos_r (X_v, \Gamma_v) \\
 -cos_r (X_u, \Gamma_u)  & 1 &  -cos_r (X_w, \Gamma_w) \\
-cos_r (X_v, \Gamma_v)  & -cos_r (X_w, \Gamma_w) & 1
\end{array} \right)$$
is positive definite then $\Gamma$ has property (T).
\end{Theorem 2}

\textbf{Structure of the paper.} Section 2 is devoted to introducing the framework developed in \cite{BS} and \cite{DJ1}, section 3 contains the vanishing results for $L^2$-cohomologies relaying on Laplacian eigenvalue, section 4 contains vanishing results for $L^2$-cohomologies relaying on cosines of links and section 5 contains examples of groups in which Theorem 1 proves property (T) (sadly, at the this time we are not able to produce examples for Theorem 2 that differ from the ones given in \cite{K}) . \\ 

\textbf{Acknowledgments.} The author wishes to thank his advisor Uri Bader, for many useful conversations and much encouragement along the way.

\section{Framework}

Here we introduce the framework constructed in \cite{BS} and \cite{DJ1}. Throughout this section $X$ is a locally finite simplicial complex of dimension $n$ such that all the links of $X$ are connected, $\Gamma$ is a locally compact, properly discontinuous, unimodular group of automorphisms of $X$ and $\rho$ is a unitary representation of $\Gamma$ on a complex Hilbert space $H$.

\subsection{general settings}

Following \cite{BS} we introduce the following notations:
\begin{enumerate}
\item For $0\leq k\leq n$, denote by $\Sigma(k)$ the set of ordered $k$-simplices (i.e. $\sigma \in \Sigma(k)$ is and ordered $k+1$-tuple of vertices) and choose a set $\Sigma(k,\Gamma)\subseteq\Sigma(k)$ of representatives
of $\Gamma$-orbits.
\item For a simplex $\sigma\in\Sigma(k)$ denote by $m(\sigma)$ the number
of $n$-simplices containing $\sigma$ without ordering, i.e. if $\sigma = (v_0,...,v_k)$ then $m(\sigma)$ is the number of different unordered $n$-simplices which contain $\lbrace v_0,...,v_k \rbrace$. We shall assume that $m(\sigma)\geq 1$ for every $\sigma$. 
\item For $\tau = (w_0,...,w_k), \sigma = (v_0,...,v_l)$ s.t. $k \leq l$, we denote $\tau \subset \sigma$ if $\lbrace w_0,...,w_k \rbrace \subset \lbrace v_0,...,v_l\rbrace$ and  $\tau \sqsubset \sigma$ if there is a \underline{monotone} function $f:\lbrace 0,...,k \rbrace \rightarrow \lbrace 0,...,l \rbrace$ such that $w_i = v_{f(i)}$.
\item For a simplex $\sigma\in\Sigma(k)$, denote by $\Gamma_{\sigma}$
the stabilizer of $\sigma$ and by $\vert \Gamma_{\sigma} \vert$ the measure of $\Gamma_{\sigma}$ with respect to the Haar measure.  
\item For $0\leq k\leq n$, denote by $C^{k}(X,\rho)$ the space of simplicial
$k$-cochains of $X$ which are twisted by $\rho$, that is, $\phi\in C^{k}(X,\rho)$
is an alternating map on ordered $k$-simplices of $X$ with values
in $H$ such that \[
\forall\gamma\in\Gamma,\forall x\in\Sigma(k),\rho(\gamma)\phi(x)=\phi(\gamma x)\]

\end{enumerate}
For $\phi\in C^{k}(X,\rho)$ define $\phi$ to be \emph{square integrable
mod $\Gamma$ }if \[
\left\Vert \phi\right\Vert ^{2}:=\sum_{\sigma\in\Sigma(k,\Gamma)}\frac{m(\sigma)}{(k+1)!\left|\Gamma_{\sigma}\right|}\left\langle \phi(\sigma),\phi(\sigma)\right\rangle <\infty\]

Denote by $L^{k}(X,\rho)$ the space of all square integrable cochains
in $C^{k}(X,\rho)$. On $L^{k}(X,\rho)$ there is an Hermitian form
given by \[
\left\langle \phi,\psi\right\rangle :=\sum_{\sigma\in\Sigma(k,\Gamma)}\frac{m(\sigma)}{(k+1)!\left|\Gamma_{\sigma}\right|}\left\langle \phi(\sigma),\psi(\sigma)\right\rangle ,\;\phi,\psi\in L^{k}(X,\rho)\]

Note that if $\Gamma$ acts cocompactly on $X$ then $L^{k}(X,\rho)=C^{k}(X,\rho)$.
To distinguish the norm of $L^{k}(X,\rho)$ from the norm of $H$
we will use $\left|.\right|$ for the norm of $H$ (i.e. $\left\langle \phi(\sigma),\phi(\sigma)\right\rangle =\left|\phi(\sigma)\right|^{2}$).

For $0\leq k<n$, the \emph{differential} $d:C^{k}(X,\rho)\rightarrow C^{k+1}(X,\rho)$
is given by \[
d\phi(\sigma):=\sum_{i=0}^{k+1}(-1)^{i}\phi(\sigma_{i}),\;\sigma\in\Sigma(k+1)\]

Where $\sigma_{i}=\left(v_{0},...,\hat{v_{i}},...,v_{k+1}\right)$
for $\left(v_{0},...,v_{k+1}\right)=\sigma\in\Sigma(k+1)$.
\begin{proposition}
In \cite{BS} the following facts are proved:
\begin{enumerate}
\item \cite[Proposition 1.5]{BS} When the differential is restricted
to $d:L^{k}(X,\rho)\rightarrow L^{k+1}(X,\rho)$ it is a bounded operator.
\item \cite[Proposition 1.6]{BS} The adjoint operator of $d$, denoted
by $\delta:L^{k+1}(X,\rho)\rightarrow L^{k}(X,\rho)$ is \[
\delta\phi(\tau)=\sum_{v\in\Sigma(0),v\tau\in\Sigma(k+1)}\frac{m(v\tau)}{m(\tau)}\phi(v\tau),\;\tau\in\Sigma(k)\]
where $v\tau=(v,v_{0},...,v_{k})$ for $\tau=(v_{0},...,v_{k})$ 
\item \cite[Corollary 1.7]{BS} For $\phi\in L^{k}(X,\rho)$ and $\sigma\in\Sigma(k)$,
\[
\delta d\phi(\sigma)=(n-k)\phi(\sigma)-\sum_{v\in\Sigma(0),v\tau\in\Sigma(k+1)}\sum_{0\leq i\leq k}(-1)^{i}\frac{m(v\sigma)}{m(\sigma)}\phi(v\sigma_{i})\]

\end{enumerate}
\end{proposition}
\begin{remark}
Note that if $X$ is a 1-dimentional connected graph, $k=0$
and $H=\mathbb{R}$, then $\delta d$ is the (classical) graph Laplacian. 
\end{remark}

\subsection{$L^{2}$-cohomology}

The $L^{2}$-cohomology of $X$ with respect to $\rho$ is defined
as \[
L^{2}H^{k}(X,\rho)=ker(d\mid L^{k}(X,\rho))/im(d\mid L^{k-1}(X,\rho))\]

And the reduced $L^{2}$-cohomology of $X$ with respect to $\rho$ is defined
as \[
ker(d\mid L^{k}(X,\rho))/\overline{im(d\mid L^{k-1}(X,\rho))}\]  

Note that if $\Gamma$ acts cocompactly on a contractible $X$ then $L^{2}H^{k}(X,\rho)=L^2 H^{k}(\Gamma,\rho)$.

Let \[
\Delta^{+}=\delta d,\Delta^{-}=d\delta,\Delta=\Delta^{+}+\Delta^{-}\]

Define the space of $L^{2}\mathcal{H}^{k}(X,\rho)$ of square integrable
mod $\Gamma$ \emph{harmonic k-forms} twisted by $\rho$ to be \[
L^{2}\mathcal{H}^{k}(X,\rho)=ker(\Delta\mid L^{k}(X,\rho))=ker(\Delta^{+}\mid L^{k}(X,\rho))\cap ker(\Delta^{-}\mid L^{k}(X,\rho))\]

We have \[
(ker(\Delta^{+}\mid L^{k}(X,\rho)))^{\bot}=\overline{im(\Delta^{+}\mid L^{k}(X,\rho))}=\overline{im(\delta\mid L^{k+1}(X,\rho))}\]

\[
(ker(\Delta^{-}\mid L^{k}(X,\rho)))^{\bot}=\overline{im(\Delta^{-}\mid L^{k}(X,\rho))}=\overline{im(d\mid L^{k-1}(X,\rho))}\]

and the orthogonal decompositions 

\[
ker(\Delta^{+}\mid L^{k}(X,\rho))=L^{2}\mathcal{H}^{k}(X,\rho)\oplus\overline{im(\Delta^{-}\mid L^{k}(X,\rho))}\]

\[
ker(\Delta^{-}\mid L^{k}(X,\rho))=L^{2}\mathcal{H}^{k}(X,\rho)\oplus\overline{im(\Delta^{+}\mid L^{k}(X,\rho))}\]

In particular, \[
L^{2}\mathcal{H}^{k}(X,\rho)=ker(d\mid L^{k}(X,\rho))/\overline{im(d\mid L^{k-1}(X,\rho))}\]

So $L^{2}\mathcal{H}^{k}(X,\rho)$ equals the reduced cohomology.

\subsection{Localization}

Let $\left(v_{0},...,v_{j}\right)=\tau\in\Sigma(j)$, denote by $X_{\tau}$
the \emph{link} of $\tau$ in $X$, that is, the subcomplex of dimension
$n-j-1$ consisting on simplices $\sigma=\left(w_{0},...,w_{k}\right)$
such that $\left\{ v_{0},...,v_{j}\right\} ,\left\{ w_{0},...,w_{k}\right\}$ are disjoint as sets
and $\left(v_{0},...,v_{j},w_{0},...,w_{k}\right)=\tau\sigma\in\Sigma(j+k+1)$.
The isotropy group $\Gamma_{\tau}$ acts by automorphisms on $X_{\tau}$
and if we denote by $\rho_{\tau}$ the restriction of $\rho$ to $\Gamma_{\tau}$,
we get that $\rho_{\tau}$ is a unitary representation of $\Gamma_{\tau}$.
Note that since $X$ is locally finite and $\Gamma$ acts properly discontinuously,
we get that $X_{\tau}$ is finite and $\Gamma_{\tau}$ is compact. Denote as in the general settings:
\begin{enumerate}
\item For $0\leq k\leq n-j-1$, denote by $\Sigma_{\tau}(k)$ the set of
ordered $k$-simplices and choose a set $\Sigma_{\tau}(k,\Gamma_{\tau})\subseteq\Sigma_\tau (k)$
of representatives of $\Gamma_{\tau}$-orbits.
\item For a simplex $\sigma\in\Sigma_{\tau}(k)$ denote by $m_{\tau}(\sigma)$
the number of $n-j-1$-simplices containing $\sigma$ in $X_{\tau}$.
Note that $m_{\tau}(\sigma)=m(\tau\sigma)$ and by our previous assumption,
$m_{\tau}(\sigma)\geq1$ for every $\sigma$. 
\item For a simplex $\sigma\in\Sigma_{\tau}(k)$, denote by $\Gamma_{\tau\sigma}$
the stabilizer of $\sigma$ in $\Gamma_{\tau}$. 
\item For $0\leq k\leq n-j-1$, denote by $C^{k}(X_{\tau},\rho_{\tau})$
the space of simplicial $k$-cochains of $X_{\tau}$ which are twisted
by $\rho_{\tau}$.
\end{enumerate}
We also define $L^{k}(X_{\tau},\rho_{\tau}),d_{\tau},\delta_{\tau},\Delta_{\tau}^{+},\Delta_{\tau}^{-},\Delta_{\tau}$
and the $L^{2}$-cohomology as before. Note that since $X_{\tau}$
is finite we get that $L^{k}(X_{\tau},\rho_{\tau})=C^{k}(X_{\tau},\rho_{\tau})$.

Define the \emph{localization map} \[
C^{k}(X,\rho)\rightarrow C^{k-j-1}(X_{\tau},\rho_{\tau}),\;\phi\rightarrow\phi_{\tau}\]

Where $\phi_{\tau}(\sigma)=\phi(\tau\sigma)$.

\subsection{Further results from \cite{BS}, \cite{DJ1}}

\begin{proposition}
\label{resultfrombs}
In \cite{BS},\cite{DJ1} following results were proven:
\begin{enumerate}

\item \cite[Lemma 1.3]{BS}, \cite[Lemma 3.3]{DJ1} For $0\leq l<k\leq n$, let $f=f(\tau,\sigma)$ be
a $\Gamma$-invariant function on the set of pairs $\left(\tau,\sigma\right)$,
where $\tau$ is an ordered $l$-simplex and $\sigma$ is an ordered
$k$-simplex with $\tau\subset\sigma$ Then \[
\sum_{\sigma\in\Sigma(k,\Gamma)}\sum_{\begin{array}{c}
{\scriptstyle \tau\in\Sigma(l)}\\
{\scriptstyle \tau\subset\sigma}\end{array}}\frac{f(\tau,\sigma)}{\left|\Gamma_{\sigma}\right|}=\sum_{\tau\in\Sigma(l,\Gamma)}\sum_{\begin{array}{c}
{\scriptstyle \sigma\in\Sigma(k)}\\
{\scriptstyle \tau\subset\sigma}\end{array}}\frac{f(\tau,\sigma)}{\left|\Gamma_{\tau}\right|}\]

\item \cite[equality (1.9)]{BS} For $\psi\in C^{l}(X_{\tau},\rho_{\tau})$ we
have \[
\left\Vert \psi\right\Vert ^{2}=\sum_{\eta\in\Sigma_{\tau}(l,\Gamma_{\tau})}\frac{m_{\tau}(\eta)}{\left|\Gamma_{\tau\eta}\right|(l+1)!}\left|\psi(\eta)\right|^{2}=\frac{1}{(l+1)!\left|\Gamma_{\tau}\right|}\sum_{\eta\in\Sigma_{\tau}(l)}m(\tau\eta)\left|\psi(\eta)\right|^{2}\]

\item \cite[Lemma 1.11]{BS} For $0\leq k\leq n$, $\tau\in\Sigma(k-1)$ and
$\phi\in L^{k}(X,\rho)$, denote by $\phi_{\tau}^{0}$ the component
of $\phi_{\tau}$ in the subspace of constant maps in $C^{0}(X_{\tau},\rho_{\tau})$
then \[
\left\Vert \phi_{\tau}^{0}\right\Vert ^{2}=\frac{m(\tau)}{(n-k+1)\left|\Gamma_{\tau}\right|}\left|\delta\phi(\tau)\right|^{2}\]
and therefore
$$ \Vert \delta \phi \Vert^2 = \sum_{\tau\in\Sigma(k-1,\Gamma)} \frac{n-k+1}{k!} \left\Vert \phi_{\tau}^{0}\right\Vert ^{2}$$

\item \cite[Theorem 1.12]{BS} Let $0 \leq j < k <n$ and $\phi \in L^k (X, \rho)$ then 
$$ k! (\Vert d \phi \Vert^2 - (n-k) \Vert \phi \Vert^2 )= (k-j-1)! \sum_{\tau \in \Sigma (j, \Gamma)} \Vert d_\tau \phi_\tau \Vert^2 - (n-k) \Vert \phi_\tau \Vert^2$$
\cite[Corollary 1.13]{BS} For $0<k<n$ and $\tau\in\Sigma(k-1)$ define
a quadratic form $Q_{\tau}$ on $C^{0}(X_{\tau},\rho_{\tau})$ by
\[
Q_{\tau}(\psi)=\left\Vert d_{\tau}\psi\right\Vert ^{2}-\frac{k}{k+1}(n-k)\left\Vert \psi\right\Vert ^{2}=\left\langle \Delta_{\tau}^{+}\psi,\psi\right\rangle -\frac{k}{k+1}(n-k)\left\Vert \psi\right\Vert ^{2}\]
Then for every $\phi\in L^{k}(X,\rho)$ we have \[
k!\left\Vert d\phi\right\Vert ^{2}=\sum_{\tau\in\Sigma(k-1,\Gamma)}Q_{\tau}(\phi_{\tau})\]

\end{enumerate}
\end{proposition}

\subsection{Vanishing of $L^{2}$-cohomologies}

The idea (taken from \cite{BS}) for proving that $L^2 H^k (X, \rho)=0$ for every $\rho$ goes as follows: prove that there is an $\varepsilon >0$ such that:
$$ \forall \phi \in L^k (X, \rho), \Vert d \phi \Vert^2 = 0 \Rightarrow \Vert \delta \phi \Vert^2 \geq \varepsilon \Vert \phi \Vert^2$$
This will imply that $ker(d\mid L^{k}(X,\rho)) \cap ker (\delta\mid L^{k}(X,\rho)) = 0$ and therefore $ker(\Delta^+\mid L^{k}(X,\rho)) \cap ker (\Delta^-\mid L^{k}(X,\rho)) = 0$ so we get $L^{2}\mathcal{H}^{k}(X,\rho)=0$. Also, by the above inequality, we get that the image of $\Delta^{-}\mid L^{k}(X,\rho)$ is closed in
$ker(\Delta^{+}\mid L^{k}(X,\rho))$. Hence $ker(\Delta^{+}\mid L^{k}(X,\rho))=im(\Delta^{-}\mid L^{k}(X,\rho))$ and therefore $L^{2}H^{k}(X,\rho)=0$.     
In particular, one could prove the vanishing of $L^k (X,\rho)$ for every $\rho$ by showing that there is an $\varepsilon >0$ such that for every representation $\rho$ the following holds:
$$ \forall \phi \in  L^k (X, \rho), \Vert d \phi \Vert^2 + \Vert \delta \phi \Vert^2 \geq \varepsilon \Vert \phi \Vert^2$$
In the same manner, to prove that $L^2 \mathcal{H}^k (X, \rho)=0$ for every $\rho$, it is enough to prove:
$$ \forall \phi \in L^k (X, \rho), \Vert \delta \phi \Vert^2 = 0 \Rightarrow \Vert d \phi \Vert^2 > 0$$

\section{Criteria via Laplacian eigenvalues}

In this section we shall find criteria for the vanishing of $L^2$-cohomologies using the eigenvalues for the Laplacians.

\begin{definition}
For $0 \leq k < l$ integers define $M_{k}^l$ as following: let $\gamma$ be an ordered simplex of dimension $l$  denote $F=\lbrace \sigma \sqsubset \gamma : dim(\sigma) =k \rbrace, F'=\lbrace \nu \sqsubset \gamma : dim( \nu )=k+1 \rbrace$. Denote by $[\nu:\sigma ]$ the sign of $\sigma$ in $\nu$, i.e. if $\nu = (v_0,...,v_{k+1}) , \sigma = (v_0,..., \hat{v_i},...,v_{k+1})$, then $[\nu : \sigma ]=(-1)^i$. Now $M_{k}^l$ is a matrix indexed by $F \cup F'$ defined as
$$  (M_{k}^l)_{\alpha, \beta } = \begin{cases}
x_\alpha & \alpha=\beta \in F \\
[\beta:\alpha ] & \alpha \in F, \beta \in F', \alpha \sqsubset \beta \\
[\alpha : \beta ] & \beta \in F, \alpha \in F', \beta \sqsubset \alpha \\
0 & otherwise
\end{cases}  $$
Define $p_k^l (x_{\sigma})$ to be the multi polynomial with variables indexed by $F$ as 
$$p_k^l (x_\sigma)=det (M_k^l)$$
\end{definition}     

\begin{example} 
For $l=k+1$, 
$$p_k^l (x_{\sigma}) = \sum_{\sigma \in F} \prod_{\sigma' \in F, \sigma' \neq D} x_{\sigma'}$$
notice that in this case we get 
$$ p_k^l (x_{\sigma}) = \nabla \cdot (\prod_{\sigma' \in F} x_{\sigma'})$$ 
\end{example}

\begin{remark}
Notice that if $2(k+1) < l+1 $ then 
$$\vert F \vert = \left( \begin{array}{c} l+1 \\ k+1 \end{array} \right) \leq \left( \begin{array}{c} l+1 \\ k+2 \end{array} \right) = \vert F' \vert$$ 
and therefore $p_k^l$ is constant (if $2(k+1) < l $ then $p_k^l \equiv 0$).
\end{remark}

For every simplex $\tau$ of dimension $k-1$ denote by $\lambda_\tau$ the smallest positive eigenvalue of $\Delta^+_\tau$ on $X_\tau$ (the link of $\tau$) also denote $\overline{\lambda_\tau} = \lambda_\tau - \frac{k(n-k)}{k+1} $. Now for a simplex $\sigma$ of dimension $k$ denote 
$$S_\sigma =  \sum_{\tau \in \Sigma(k-1), \tau \sqsubset \sigma} \overline{\lambda_\tau} $$ 

\begin{theorem}
If there is an $l>k$ and $\varepsilon >0$ such that for any $\gamma \in \Sigma(l,\Gamma)$ one has
$$p_k^l (\lambda-S_\sigma) = 0 \Rightarrow \lambda \geq \varepsilon$$
where $p_k^l (\lambda-S_\sigma)$ is the polynomial obtained by placing $x_{\sigma} = \lambda - S_\sigma$,
then $L^2 H^k(X,\rho)=0$ for every $\rho$. 
\end{theorem}

\begin{proof}
We shall start with repeating the proof in \cite{BS}: \\
Let $\phi \in L^k(X,\rho)$ then by Proposition \ref{resultfrombs} (4), one has
$$k!\left\Vert d\phi\right\Vert ^{2} = \sum_{\tau\in\Sigma(k-1,\Gamma)} \langle \Delta^+_\tau (\phi_{\tau}),\phi_{\tau} \rangle - \frac{k}{k+1}(n-k)\left\Vert \phi_\tau \right\Vert^{2}$$
Now denote as before $\phi_{\tau}^0$ to be the projection of $\phi_{\tau}$ on the space of constant function on $X_\tau$ and by $\phi_{\tau}^1$ its orthogonal complement. Note that since $ker(\Delta^+_\tau)$ is that space of constant function on $X_\tau$ (since $X_\tau$ is connected), we get that 
$$\langle \Delta^+_\tau (\phi_{\tau}),\phi_{\tau} \rangle = \langle \Delta^+_\tau (\phi_{\tau}^1),\phi_{\tau}^1 \rangle \geq \lambda_\tau \Vert \phi_{\tau}^1 \Vert^2$$  
Therefore
$$k!\left\Vert d\phi\right\Vert ^{2} \geq \sum_{\tau\in\Sigma(k-1,\Gamma)} \lambda_\tau \Vert \phi_{\tau}^1 \Vert^2 -  \frac{k}{k+1}(n-k)\left\Vert \phi_\tau \right\Vert^{2} = $$
$$ = \sum_{\tau\in\Sigma(k-1,\Gamma)} \lambda_\tau \Vert \phi_{\tau} \Vert^2 - \lambda_\tau \Vert \phi_{\tau}^0 \Vert^2 -  \frac{k}{k+1}(n-k)\left\Vert \phi_\tau \right\Vert^{2}$$
So we get 
$$ k!\left\Vert d\phi\right\Vert^{2} + \sum_{\tau\in\Sigma(k-1,\Gamma)} \lambda_\tau \Vert \phi_{\tau}^0 \Vert^2 \geq \sum_{\tau\in\Sigma(k-1,\Gamma)} \overline{\lambda_\tau}  \left\Vert \phi_\tau \right\Vert^{2}$$   
Note that since $\Vert d_\tau \Vert^2 \leq 2(n-k)$ for every $\tau$ (see \cite[Proposition 1.5]{BS}) then $\lambda_\tau \leq 2(n-k)$ for every $\tau$ and therefore 
$$k! \frac{2(n-k)}{n-k+1} \Vert \delta \phi \Vert^2 = \sum_{\tau\in\Sigma(k-1,\Gamma)} 2(n-k) \Vert \phi_{\tau}^0 \Vert^2 \geq \sum_{\tau\in\Sigma(k-1,\Gamma)} \lambda_\tau \Vert \phi_{\tau}^0 \Vert^2$$
So we get 
$$ k!\left\Vert d\phi\right\Vert^{2} + k! \frac{2(n-k)}{n-k+1} \Vert \delta \phi \Vert^2 \geq \sum_{\tau\in\Sigma(k-1,\Gamma)} \overline{\lambda_\tau}  \left\Vert \phi_\tau \right\Vert^{2}$$

\begin{remark}
At this stage, if we wanted to prove the result in \cite{BS} (every $\overline{\lambda_\tau} \geq \varepsilon$ implies the vanishing of the cohomology) we would be done, because then
$$ k!\left\Vert d\phi\right\Vert^{2} + k! \frac{2(n-k)}{n-k+1} \Vert \delta \phi \Vert^2 \geq \varepsilon (k+1)! \Vert \phi \Vert^2$$
\end{remark}

Now we shall assume that $\phi \in ker(d)$ and show that under the condition stated in the theorem we get that 
$$  \frac{2(n-k)}{n-k+1} \Vert \delta \phi \Vert^2 \geq \varepsilon \Vert \phi \Vert^2$$
and that will finish the proof at stated at the beginning of this section.

By definition and Proposition \ref{resultfrombs} (2) we get that 
$$ \sum_{\tau\in\Sigma(k-1,\Gamma)} \overline{\lambda_\tau}  \left\Vert \phi_\tau \right\Vert^{2} = \sum_{\tau\in\Sigma(k-1,\Gamma)} \overline{\lambda_\tau} \frac{1}{\vert \Gamma_\tau \vert} \sum_{\eta \in\Sigma_{\tau}(0)}m(\tau\eta)\left|\phi(\tau \eta)\right|^{2} = $$
$$ = \sum_{\tau\in\Sigma(k-1,\Gamma)} \overline{\lambda_\tau} \frac{1}{\vert \Gamma_\tau \vert} \sum_{\sigma  \in\Sigma (k),\tau \subset \sigma } \frac{1}{(k+1)!} m(\sigma)\left|\phi(\sigma)\right|^{2}$$
by Proposition \ref{resultfrombs} (1) we can change the order of summation and get
$$\sum_{\sigma \in\Sigma(k,\Gamma)} \frac{m(\sigma) \vert \phi(\sigma) \vert^{2}}{(k+1)! \vert \Gamma_\sigma \vert} \sum_{\tau  \in\Sigma (k-1),\tau \subset \sigma } \overline{\lambda_\tau} = k! \sum_{\sigma \in\Sigma(k,\Gamma)} \frac{m(\sigma) \vert \phi(\sigma) \vert^{2}}{(k+1)! \vert \Gamma_\sigma \vert} S_\sigma $$
Now note that 
$$(l+1)! \left( \begin{array}{c}  n-k \\ l-k  \end{array} \right) m(\sigma) = \sum_{\gamma \in \Sigma (l), \sigma \subset \gamma} m(\gamma) $$  
so we get
$$ k! \sum_{\sigma \in\Sigma(k,\Gamma)} \frac{m(\sigma) \vert \phi(\sigma) \vert^{2}}{(k+1)! \vert \Gamma_\sigma \vert} S_\sigma = k! \sum_{\sigma \in\Sigma(k,\Gamma)} \frac{\vert \phi(\sigma) \vert^{2}}{(k+1)! \vert \Gamma_\sigma \vert} S_\sigma \sum_{\gamma \in \Sigma (l), \sigma \subset \gamma} \frac{m(\gamma )}{(l+1)! \left( \begin{array}{c}  n-k \\ l-k  \end{array} \right) }$$
and by changing the order of summation again we get
$$\frac{k!}{(k+1)!} \sum_{\gamma \in\Sigma(l,\Gamma)} \frac{m(\gamma )}{(l+1)! \left( \begin{array}{c}  n-k \\ l-k  \end{array} \right)  \vert \Gamma_\gamma \vert} \sum_{\sigma \in \Sigma (k), \sigma \subset \gamma} S_\sigma \vert \phi (\sigma) \vert^2 $$

Now we shall show that, under the conditions of the theorem, for every $\gamma \in\Sigma(l,\Gamma)$ we have the the following inequality 
$$\sum_{\sigma \in \Sigma (k), \sigma \subset \gamma} S_\sigma \vert \phi (\sigma) \vert^2 \geq \varepsilon   \sum_{\sigma \in \Sigma (k), \sigma \subset \gamma} \vert \phi (\sigma) \vert^2$$
and that will complete the proof because then we have 
$$\frac{k!}{(k+1)!} \sum_{\gamma \in\Sigma(l,\Gamma)} \frac{m(\gamma )}{(l+1)! \left( \begin{array}{c}  n-k \\ l-k  \end{array} \right)  \vert \Gamma_\gamma \vert} \sum_{\sigma \in \Sigma (k), \sigma \subset \gamma} S_\sigma \vert \phi (\sigma) \vert^2 \geq $$

$$\geq \frac{k!}{(k+1)!} \sum_{\gamma \in\Sigma(l,\Gamma)} \frac{m(\gamma )}{(l+1)! \left( \begin{array}{c}  n-k \\ l-k  \end{array} \right)  \vert \Gamma_\gamma \vert} \sum_{\sigma \in \Sigma (k), \sigma \subset \gamma} \varepsilon \vert \phi (\sigma) \vert^2 = $$

$$= \varepsilon \frac{k!}{(k+1)!} \sum_{\sigma \in\Sigma(k,\Gamma)}  \frac{\vert \phi (\sigma) \vert^2 }{\vert \Gamma_\sigma \vert} \sum_{\gamma \in \Sigma (l), \sigma \subset \gamma}  \frac{m(\gamma )}{(l+1)! \left( \begin{array}{c}  n-k \\ l-k  \end{array} \right)} = $$

$$= \varepsilon k! \sum_{\sigma \in\Sigma(k,\Gamma)}  \frac{ m( \sigma) \vert \phi (\sigma) \vert^2 }{(k+1)! \vert \Gamma_\sigma \vert} = \varepsilon k! \Vert \phi \Vert^2 $$

And therefore we get
$$ \frac{2(n-k)}{n-k+1} \Vert \delta \phi \Vert^2 \geq \varepsilon \Vert \phi \Vert^2$$

So we are left with proving the following inequality - for every $\phi \in L^k(X,\rho) \cap ker(d)$ and for every $\gamma \in \Sigma (l,\Gamma)$ one has (under the conditions of the theorem):
$$\sum_{\sigma \in \Sigma (k), \sigma \subset \gamma} S_\sigma \vert \phi (\sigma) \vert^2 \geq \varepsilon   \sum_{\sigma \in \Sigma (k), \sigma \subset \gamma} \vert \phi (\sigma) \vert^2$$
Fix $\gamma \in \Sigma (l,\Gamma)$, first note that since $\phi$ is alternating we get that 
$$\sum_{\sigma \in \Sigma (k), \sigma \subset \gamma} S_\sigma \vert \phi (\sigma) \vert^2 = (k+1)! \sum_{\sigma \in \Sigma (k), \sigma \sqsubset \gamma} S_\sigma \vert \phi (\sigma) \vert^2$$
and
$$\sum_{\sigma \in \Sigma (k), \sigma \subset \gamma} \vert \phi (\sigma) \vert^2 = (k+1)! \sum_{\sigma \in \Sigma (k), \sigma \sqsubset \gamma} \vert \phi (\sigma) \vert^2$$ 
therefore it is enough to prove that 
$$\sum_{\sigma \in \Sigma (k), \sigma \sqsubset \gamma} S_\sigma \vert \phi (\sigma) \vert^2 \geq \varepsilon \sum_{\sigma \in \Sigma (k), \sigma \sqsubset \gamma} \vert \phi (\sigma) \vert^2$$ 

Now we shall need the following simple but useful lemma, which is a straightforward generalization of \cite[Lemma 2.3]{BS}. 

\begin{lemma}
For a finite set $F=\lbrace x_1,...,x_n \rbrace$, denote by $C(F,H)$ the space of maps from $F$ to $H$. 
On that space there is a natural inner product:
$$\langle \phi , \psi \rangle = \sum_{k=1}^n \langle \phi (x_k) , \psi (x_k) \rangle$$ 
and denote by $\Vert . \Vert$ the norm induced by this inner product.
Note that $C(F,H)=C(F,\mathbb{R}) \otimes H$ (as a Hilbert space). 
Let $T_j, j=1,...,m$ and $M_i, i=1,...,l$ be bounded linear operators on $C(F,H)$ so there are operators $T_{j,\mathbb{R}},M_{i,\mathbb{R}}$ on $C(F,\mathbb{R})$ such that $T_j=T_{j,\mathbb{R}} \otimes id, M_i=M_{i,\mathbb{R}} \otimes id$. 
Let $a_1,...,a_j \in \mathbb{R}$ be constants, then if there is an $\varepsilon$ such that 
$$\forall \phi \in C(F,\mathbb{R}), \phi \in \cap ker(M_{i,\mathbb{R}}) \Rightarrow \sum_j a_j \Vert T_{j,\mathbb{R}}\phi \Vert^2 \geq \varepsilon \Vert \phi \Vert^2$$ 
Then also
$$\forall \phi \in C(F,H), \phi \in \cap ker(M_i) \Rightarrow \sum_j a_j \Vert T_j \phi \Vert^2 \geq \varepsilon \Vert \phi \Vert^2$$    
 
\end{lemma}  

\begin{proof}
Choose $\lbrace v_\alpha \rbrace$ orthonormal basis of $H$. For every $\phi \in C(F,H)$ one has $\phi_\alpha \in C(X,\mathbb{R})$ such that $\phi = \sum \phi_\alpha \otimes v_\alpha$. Now note that $\phi \in \cap ker(M_i)$ iff $\forall \alpha, \phi_\alpha \in \cap ker(M_{i,\mathbb{R}})$, so we get 
$$\forall \phi \in \cap ker(M_i), \sum_j a_j \Vert T_j \phi \Vert^2 = \Vert \sum_j T_j \sum_\alpha \phi_\alpha \otimes v_\alpha \Vert^2 = \sum_\alpha \sum_j a_j \Vert T_{j,\mathbb{R}} \phi_\alpha \Vert^2 \geq $$
$$ \geq \sum_\alpha  \varepsilon \Vert \phi_\alpha \Vert^2 = \varepsilon \Vert \phi \Vert^2$$
\end{proof}

Now we can use the above lemma to reduce our problem to a much simpler one. Denote $F=\lbrace \sigma \in \Sigma (k) : \sigma \sqsubset \gamma \rbrace$ and $F' = \lbrace \nu \in \Sigma (k+1) : \nu \sqsubset \gamma \rbrace$ and for $\phi \in L^k (X,\rho) \cap ker(d)$ we can look at the restriction of $\phi$ to $F$. Note that $\phi \in ker(d)$ implies that for every $\nu \in \Sigma (k+1)$ we have $d\phi ( \nu)=0$ and in particular $ \forall \nu \in F', d \phi (\nu)=0 $. Therefore for every $\nu \in F'$, we define $M_\nu$ acting on $C(F,H)$ as $M_\nu \phi = \sum (-1)^i \phi (\nu_i)$. For every $\sigma \in F$ define $T_\sigma$ acting on $C(F,H)$ as the projection on the space spanned by the indicator function of $\sigma$. 
If we can prove that for every $\phi \in C(F,H)$ the following holds 
$$ \phi \in \cap ker (M_\nu) \Rightarrow \sum_{\sigma \sqsubset \gamma} S_\sigma \Vert T_\sigma \phi \Vert^2 \geq \varepsilon \Vert \phi \Vert^2$$
then in particular, for every $\phi \in C^k (X,\rho)$ we have 
$$\sum_{\sigma \in \Sigma (k), \sigma \sqsubset \gamma} S_\sigma \vert \phi (\sigma) \vert^2 \geq \varepsilon \sum_{\sigma \in \Sigma (k), \sigma \sqsubset \gamma} \vert \phi (\sigma) \vert^2$$     

By the above lemma, it is enough to prove that for $\phi \in C(F,\mathbb{R})$ we have 
$$ \phi \in \cap ker (M_\nu) \Rightarrow \sum_{\sigma \sqsubset \gamma} S_\sigma \Vert T_\sigma \phi \Vert^2 \geq \varepsilon \Vert \phi \Vert^2$$
and if we denote $\phi (\sigma) = x_\sigma$ we get the following problem: prove that 
$$ \sum_{\sigma \in \Sigma (k), \sigma \sqsubset \gamma} S_\sigma x_\sigma^2 \geq \varepsilon \sum_{\sigma \in \Sigma (k), \sigma \sqsubset \gamma} x_\sigma^2$$
under the constraints
$$\forall \nu \in F', \sum_{i=0}^{k+1} (-1)^i x_{\nu_i} =0$$
since both sides of the inequality we are trying to prove are quadratic, WLOG it is enough to prove that
$$ \sum_{\sigma \in \Sigma (k), \sigma \sqsubset \gamma} S_\sigma x_\sigma^2 \geq \varepsilon$$
under the constraints 
$$ \sum_{\sigma \in \Sigma (k), \sigma \sqsubset \gamma} x_\sigma^2 =1$$
$$\forall \nu \in F', \sum_{\sigma \in F, \sigma \sqsubset \nu} [\nu : \sigma] x_{\sigma} =0$$

This is a problem of finding a minimum of a function in $\mathbb{R}^{\vert F \vert}$ under constraints which define a compact set in $\mathbb{R}^{\vert F \vert}$ and so we can use the Lagrange multiplier theorem.
Define the Lagrange function to be
$$\Lambda (x_\sigma,\mu_\nu, \lambda) = \sum_{\sigma \in \Sigma (k), \sigma \sqsubset \gamma} S_\sigma x_\sigma^2 - 2 \sum_{\nu \in F'} \mu_\nu (\sum_{\sigma \in F, \sigma \sqsubset \nu} [\nu : \sigma ] x_{\sigma} - \lambda (\sum_{\sigma \in \Sigma (k), \sigma \sqsubset \gamma} x_\sigma^2 -1)$$
(the $2$ multiplying $\sum_{\nu \in F'}$ is added for convenience) 

So for every $\sigma \in F$ we get a equation by derivation of $\Lambda$ by $x_\sigma$: 
$$ 2S_\sigma x_\sigma - 2 \lambda x_\sigma - 2 \sum_{\nu \in F', \sigma \sqsubset \nu} [\nu : \sigma ] \mu_\nu = 0$$
if we multiply every such equation by $x_\sigma$ and add them up (over all $\sigma \in F$) we get
$$2 \sum_{\sigma \in \Sigma (k), \sigma \sqsubset \gamma} S_\sigma x_\sigma^2 - 2 \lambda (\sum_{\sigma \in \Sigma (k), \sigma \sqsubset \gamma} x_\sigma^2) - 2 \sum_{\nu \in F'} \mu_\nu \sum_{\sigma \in F, \sigma \sqsubset \nu} [\nu : \sigma ] x_{\sigma}=0$$
and by the equations coming from the constraints we get 
$$ \sum_{\sigma \in \Sigma (k), \sigma \sqsubset \gamma} S_\sigma x_\sigma^2 = \lambda$$
So the minimum is must some $\lambda$ which is a part of a vector $(x_\sigma,\mu_\nu,\lambda)$ which solves
$ \nabla \Lambda =0$.

Treat $\lambda$ as a parameter and consider the system of linear equations in $(x_\sigma,\mu_\nu)$
$$ (\lambda - S_\sigma ) x_\sigma +  \sum_{\nu \in F', \sigma \sqsubset \nu} [\nu : \sigma ] \mu_\nu = 0$$
$$ \sum_{\sigma \in F, \sigma \sqsubset \nu} [\nu : \sigma ] x_{\sigma} =0$$
Note that from  
$$ \sum_{\sigma \in \Sigma (k), \sigma \sqsubset \gamma} x_\sigma^2 =1$$
the minimum is obtain only if this system of equations have a non trivial solution, that is only if the determinant is zero, but this determinant is exactly $p_k^l (\lambda-S_\sigma)$ and by the conditions of the theorem we know that for every root $\lambda$ of this polynomial we have $\lambda \geq \varepsilon$ and so we are done.

\end{proof}

Now Theorem 1 is proven as a corollary:

\begin{corollary}
\label{T for 2d}
Let $\Gamma$ be a locally compact, properly discontinuous, unimodular group of automorphisms
of $X$ acting cocompactly on $X$ such that $X$ is a locally finite contractible $2$-dimensional simplicial complex. If for every $(u,v,w) \in \Sigma (2, \Gamma)$ the following holds:
\begin{itemize}
\item $ \overline{\lambda_u} + \overline{\lambda_v} + \overline{\lambda_w} >0$
\item $ S_{(u,v)}S_{(u,w)} + S_{(u,v)}S_{(v,w)} + S_{(u,w)}S_{(v,w)} >0$
\end{itemize}
then $\Gamma$ has property (T).
\end{corollary}

\begin{proof}
By the above theorem it is enough to prove that for every $(u,v,w) \in \Sigma (2, \Gamma)$ all the roots of $p_1^2 (\lambda - S_{(u,v)} , \lambda - S_{(v,w)}, \lambda - S_{(u,w)} )$ are positive. We have
$$p_1^2 (\lambda - S_{(u,v)} , \lambda - S_{(v,w)}, \lambda - S_{(u,w)} ) = $$ 
$$ =(\lambda - S_{(u,v)})( \lambda - S_{(v,w)}) + (\lambda - S_{(u,v)})(\lambda - S_{(u,w)}) + (\lambda - S_{(v,w)})(\lambda - S_{(u,w)}) $$
one can find the roots of the polynomial explicitly and the roots of this polynomial are positive iff
\begin{itemize}
\item $ \overline{\lambda_u} + \overline{\lambda_v} + \overline{\lambda_w} >0$
\item $ S_{(u,v)}S_{(u,w)} + S_{(u,v)}S_{(v,w)} + S_{(u,w)}S_{(v,w)} >0$
\end{itemize}

\end{proof}

\section{Criteria via cosine of links}

We begin with a general definition for a cosine of a finite simplicial complex with respect to a group of automorphisms.

\begin{definition}
Let $Y$ be a connected finite simplicial complex, $\Lambda$ be a subgroup of $Aut(Y)$ and $\rho$ be a unitary representation of $\Lambda$. Define the cosine of $Y$ with respect to $\Lambda$ and $\rho$ to be 
$$ cos (Y,\Lambda,\rho) = sup \lbrace \dfrac{\sum_{(u,v) \in \Sigma (1, \Lambda)} m((u,v)) \langle \phi^1 (u) , \phi^1 (v) \rangle }{\sum_{(u,v) \in \Sigma (1, \Lambda)} m((u,v)) \vert \phi (u)\vert \vert \phi (v) \vert}: 0 \neq \phi \in C^0 (Y,\rho)  \rbrace$$
where $\phi^1$ is (as above) the projection of $\phi$ on the orthogonal compliment of the space of constant functions.
Define also
$$ cos (Y,\Lambda) = sup \lbrace cos (Y,\Lambda,\rho) : {\rho \text{ is a unitary representation}} \rbrace$$     
\end{definition}

\begin{remark}
Note that $\Lambda_1 < \Lambda_2$ implies that $cos (Y,\Lambda_1) \geq cos (Y,\Lambda_2)$ and in particular if we denote $cos (Y) = cos (Y,\lbrace e \rbrace)$ then $cos (Y) \geq cos (Y,\Lambda)$ for every $\Lambda$.
\end{remark}

Now we return to our general setting, i.e. $X$ is a simplicial complex, $\Gamma$ is a locally compact unimodular group of automorphisms of $X$ acting properly discontinuously on $X$ and so on.

\begin{definition}
For $\gamma \in \Sigma (k+1)$ define $A(\gamma, \Gamma, \rho)$ to be a matrix indexed by $\lbrace \sigma \in \Sigma (k) : \sigma \sqsubset \gamma \rbrace$ as 
$$ (A(\gamma,\Gamma,\rho))_{\alpha, \beta} = \begin{cases}
1 \: ; \: \alpha = \beta \\
-cos (X_{\alpha \cap \beta},\Gamma_{\alpha \cap \beta},\rho) \: ; \: \alpha \neq \beta
\end{cases}$$
where $\alpha \cap \beta$ is the simplex of dimension $k-1$ that is contained in $\alpha$ and in $\beta$.
In the same way, define 
$$ (A(\gamma,\Gamma))_{\alpha, \beta} = \begin{cases}
1 \: ; \: \alpha = \beta \\
-cos (X_{\alpha \cap \beta},\Gamma_{\alpha \cap \beta}) \: ; \: \alpha \neq \beta
\end{cases}$$
Denote $A(\gamma)=A(\gamma, \lbrace e \rbrace)$.
\end{definition}

\begin{theorem}
If for every unitary representation $\rho$ there is an $\varepsilon >0$ such that for every $\gamma \in \Sigma (k+1,\Gamma)$ the matrix $A (\gamma,\Gamma,\rho )$ is positive definite with eigenvalues greater or equal to $\varepsilon$ then $L^2 H^k(X,\rho)=0$ for every $\rho$. 
  
\end{theorem}

\begin{proof}
Fix $\rho$ and $\phi \in L^k (X,\rho)$. For every $\gamma \in \Sigma (k+1,\Gamma)$ define the vector $V_\gamma = (\vert \phi (\sigma) \vert : \sigma \in \Sigma (k) , \sigma \sqsubset \gamma)$ so by the conditions of the theorem we get
$$ V_\gamma^t A (\gamma,\Gamma,\rho) V_\gamma \geq \varepsilon \sum_{\sigma \in \Sigma (k), \sigma \sqsubset \gamma} \vert \phi (\sigma) \vert^2 $$
which implies
$$\sum_{\gamma \in \Sigma (k+1,\Gamma)} \dfrac{m(\gamma )}{(k+2)! \vert \Gamma_\gamma \vert}  V_\gamma^t A (\gamma,\Gamma,\rho) V_\gamma \geq \varepsilon \sum_{\gamma \in \Sigma (k+1,\Gamma)} \dfrac{m(\gamma)}{\vert \Gamma_\gamma \vert (k+2)!(k+1)!}     \sum_{\sigma \in \Sigma (k), \sigma \subset \gamma} \vert \phi (\sigma) \vert^2$$
changing the order of summation in the right hand side of the inequality yields
$$ \varepsilon \sum_{\sigma \in \Sigma (k,\Gamma)} \dfrac{\vert \phi (\sigma) \vert^2}{\vert \Gamma_\sigma \vert (k+2)! (k+1)!}   \sum_{\gamma \in \Sigma (k+1), \sigma \subset \gamma}  m(\gamma) = $$
$$ =  \varepsilon (n-k) \sum_{\sigma \in \Sigma (k,\Gamma)} \dfrac{ m(\sigma) \vert \phi (\sigma) \vert^2}{\vert \Gamma_\sigma \vert (k+1)!} =  \varepsilon (n-k) \Vert \phi \Vert^2 $$
Now we shall work on the left hand side of the above inequality, i.e. on 
$$\sum_{\gamma \in \Sigma (k+1,\Gamma)} \dfrac{m(\gamma )}{(k+2)! \vert \Gamma_\gamma \vert}  V_\gamma^t A (\gamma,\Gamma,\rho) V_\gamma$$
writing the term $V_\gamma^t A (\gamma,\Gamma,\rho) V_\gamma$ explicitly gives 
$$ \sum_{\gamma \in \Sigma (k+1,\Gamma)} \dfrac{m(\gamma )}{ (k+2)! \vert \Gamma_\gamma \vert} \sum_{\sigma \in \Sigma (k), \sigma \sqsubset \gamma} \vert \phi (\sigma) \vert^2 - $$
$$ - \sum_{\gamma \in \Sigma (k+1,\Gamma)} \dfrac{m(\gamma )}{(k+2)! \vert \Gamma_\gamma \vert} \sum_{ \begin{scriptsize}
\begin{array}{c} \sigma,\sigma' \in \Sigma (k) \\ \sigma,\sigma' \sqsubset \gamma, \sigma \neq \sigma' \end{array} \end{scriptsize}} cos(X_{\sigma \cap \sigma'} ,\Gamma_{\sigma \cap \sigma'},\rho) \vert \phi (\sigma) \vert \vert \phi (\sigma') \vert$$

First we look at 
$$  \sum_{\gamma \in \Sigma (k+1,\Gamma)} \dfrac{m(\gamma )}{(k+2)! \vert \Gamma_\gamma \vert} \sum_{\sigma \in \Sigma (k), \sigma \sqsubset \gamma} \vert \phi (\sigma) \vert^2 = $$
$$ = \sum_{\gamma \in \Sigma (k+1,\Gamma)} \dfrac{m(\gamma )}{(k+2)! (k+1)! \vert \Gamma_\gamma \vert} \sum_{\sigma \in \Sigma (k), \sigma \subset \gamma} \vert \phi (\sigma) \vert^2 =$$
$$ = \sum_{\sigma \in \Sigma (k,\Gamma)} \dfrac{\vert \phi (\sigma) \vert^2}{(k+2)! (k+1)! \vert \Gamma_\sigma \vert} \sum_{\gamma \in \Sigma (k+1), \sigma \subset \gamma} m(\gamma) = $$
$$ (n-k) (k+2)! \sum_{\sigma \in \Sigma (k,\Gamma)} \dfrac{ m(\sigma ) \vert \phi (\sigma) \vert^2}{(k+2)!(k+1)! \vert \Gamma_\sigma \vert} = (n-k) \Vert \phi \Vert^2 $$

Now we shall look at 
$$ \sum_{\gamma \in \Sigma (k+1,\Gamma)} \dfrac{m(\gamma )}{(k+2)! \vert \Gamma_\gamma \vert}  \sum_{\begin{scriptsize} \begin{array}{c} \sigma,\sigma' \in \Sigma (k) \\ \sigma,\sigma' \sqsubset \gamma, \sigma \neq \sigma' \end{array} \end{scriptsize}} cos(X_{\sigma \cap \sigma'} ,\Gamma_{\sigma \cap \sigma'},\rho) \vert \phi (\sigma) \vert \vert \phi (\sigma') \vert = $$
$$ = \sum_{\gamma \in \Sigma (k+1,\Gamma)} \dfrac{m(\gamma )}{(k+2)! \vert \Gamma_\gamma \vert}  \sum_{\tau \in \Sigma (k-1), \tau \sqsubset \gamma} cos(X_{\tau} ,\Gamma_{\tau},\rho) \sum_{\begin{scriptsize} \begin{array}{c} u,v \in \Sigma (0) \\ \tau u, \tau v \subset \gamma \end{array} \end{scriptsize}}  \vert \phi (\tau u) \vert \vert \phi (\tau v) \vert = $$
$$ = \dfrac{1}{k!} \sum_{\gamma \in \Sigma (k+1,\Gamma)} \dfrac{m(\gamma )}{(k+2)! \vert \Gamma_\gamma \vert}  \sum_{\tau \in \Sigma (k-1), \tau \subset \gamma} cos(X_{\tau} ,\Gamma_{\tau},\rho) \sum_{\begin{scriptsize} \begin{array}{c} u,v \in \Sigma (0) \\ \tau u, \tau v \subset \gamma \end{array} \end{scriptsize}}  \vert \phi (\tau u) \vert \vert \phi (\tau v) \vert = $$
$$ = \dfrac{1}{k!} \sum_{\tau \in \Sigma (k-1,\Gamma)} \dfrac{cos(X_{\tau} ,\Gamma_{\tau},\rho)}{(k+2)! \vert \Gamma_\tau \vert}  \sum_{\gamma \in \Sigma (k+1), \tau \subset \gamma} m(\gamma) \sum_{\begin{scriptsize} \begin{array}{c} u,v \in \Sigma (0) \\ \tau u, \tau v \subset \gamma \end{array} \end{scriptsize}}  \vert \phi (\tau u) \vert \vert \phi (\tau v) \vert= $$
$$ = \dfrac{1}{k!} \sum_{\tau \in \Sigma (k-1,\Gamma)} \dfrac{cos(X_{\tau} ,\Gamma_{\tau},\rho)}{\vert \Gamma_\tau \vert}  \sum_{(u,v) \in \Sigma_\tau (1)} m_\tau ((u,v)) \vert \phi_\tau (u) \vert \vert \phi_\tau (v) \vert $$
By the definition of $cos(X_{\tau} ,\Gamma_{\tau},\rho)$ we get 
$$\frac{1}{k!} \sum_{\tau \in \Sigma (k-1,\Gamma)} \dfrac{cos(X_{\tau} ,\Gamma_{\tau},\rho)}{\vert \Gamma_\tau \vert}  \sum_{(u,v) \in \Sigma_\tau (1)} m_\tau ((u,v)) \vert \phi_\tau (u) \vert \vert \phi_\tau (v) \vert \geq$$
$$ \geq \frac{1}{k!} \sum_{\tau \in \Sigma (k-1,\Gamma)} \dfrac{1}{\vert \Gamma_\tau \vert}  \sum_{(u,v) \in \Sigma_\tau (1)} m_\tau ((u,v)) \langle \phi_\tau^1 (u) , \phi_\tau^1 (v) \rangle $$
Now note that 
$$\sum_{(u,v) \in \Sigma_\tau (1)} m_\tau ((u,v)) \phi_\tau^1 (u) = \sum_{u \in \Sigma_\tau (0)} \phi_\tau^1 (u) \sum_{v \in \Sigma_\tau (0), (u,v) \in \Sigma_\tau (1)} m_\tau ((u,v)) = $$
$$ = (n-k) \sum_{u \in \Sigma_\tau (0)} m_\tau (u) \phi_\tau^1 (u) = 0$$
and 
$$ \dfrac{1}{\vert \Gamma_\tau \vert} \sum_{(u,v) \in \Sigma_\tau (1)} m_\tau ((u,v)) \vert \phi_\tau^0 \vert^2  = \dfrac{1}{\vert \Gamma_\tau \vert} \sum_{u \in \Sigma_\tau (0)} \vert \phi_\tau^0 \vert^2 \sum_{v \in \Sigma_\tau (0), (u,v) \in \Sigma_\tau (1)} m_\tau ((u,v)) = $$
$$ = \dfrac{n-k}{\vert \Gamma_\tau \vert} \sum_{u \in \Sigma_\tau (0)} m_\tau (u) \vert \phi_\tau^0 \vert^2 = (n-k) \Vert \phi_\tau^0 \Vert^2$$
and therefore
$$ \dfrac{1}{\vert \Gamma_\tau \vert} \sum_{(u,v) \in \Sigma_\tau (1)} m_\tau ((u,v)) \langle \phi_\tau (u) , \phi_\tau (v) \rangle = $$
$$=\dfrac{1}{\vert \Gamma_\tau \vert} \sum_{(u,v) \in \Sigma_\tau (1)} m_\tau ((u,v)) \langle \phi_\tau^0 +  \phi_\tau^1 (u) , \phi_\tau^0 +  \phi_\tau^1(v) \rangle =  $$
$$ = (n-k) \Vert \phi_\tau^0 \Vert^2 + \dfrac{1}{\vert \Gamma_\tau \vert} \sum_{(u,v) \in \Sigma_\tau (1)} m_\tau ((u,v)) \langle \phi_\tau^1 (u) , \phi_\tau^1(v) \rangle$$
so we get 
$$\frac{1}{k!} \sum_{\tau \in \Sigma (k-1,\Gamma)} \dfrac{1}{\vert \Gamma_\tau \vert}  \sum_{(u,v) \in \Sigma_\tau (1)} m_\tau ((u,v)) \langle \phi_\tau^1 (u) , \phi_\tau^1 (v) \rangle = $$
$$ = \frac{1}{k!} \sum_{\tau \in \Sigma (k-1,\Gamma)} - (n-k) \Vert \phi_\tau^0 \Vert^2 + \dfrac{1}{\vert \Gamma_\tau \vert}  \sum_{(u,v) \in \Sigma_\tau (1)} m_\tau ((u,v)) \langle \phi_\tau (u) , \phi_\tau (v) \rangle = $$  
$$=  - \dfrac{n-k}{n-k+1} \Vert \delta \phi \Vert^2 + \frac{1}{k!} \sum_{\tau \in \Sigma (k-1,\Gamma)} \ \dfrac{1}{\vert \Gamma_\tau \vert}  \sum_{(u,v) \in \Sigma_\tau (1)} m_\tau ((u,v)) \langle \phi_\tau (u) , \phi_\tau (v) \rangle$$
Now we shall look at 
$$\frac{1}{k!} \sum_{\tau \in \Sigma (k-1,\Gamma)}  \dfrac{1}{\vert \Gamma_\tau \vert}  \sum_{(u,v) \in \Sigma_\tau (1)} m_\tau ((u,v)) \langle \phi_\tau (u) , \phi_\tau (v) \rangle = $$
$$ \frac{1}{k!} \sum_{\tau \in \Sigma (k-1,\Gamma)}  \dfrac{1}{\vert \Gamma_\tau \vert}  \sum_{(u,v) \in \Sigma_\tau (1)} [ m_\tau ((u,v)) (-\frac{1}{2}\vert d_\tau \phi_\tau ((u,v)) \vert^2 ) +$$
$$ + \frac{1}{2} m_\tau ((u,v)) \vert \phi (u) \vert^2 + \frac{1}{2} m_\tau ((u,v)) \vert \phi (u) \vert^2] = $$
$$ = \frac{1}{k!} \sum_{\tau \in \Sigma (k-1,\Gamma)} - \Vert d_\tau \phi_\tau \Vert^2 + \frac{1}{k!} \sum_{\tau \in \Sigma (k-1,\Gamma)}  \dfrac{1}{\vert \Gamma_\tau \vert}  \sum_{(u,v) \in \Sigma_\tau (1)}  m_\tau ((u,v)) \vert \phi (u) \vert^2 = $$
$$= \frac{1}{k!} \sum_{\tau \in \Sigma (k-1,\Gamma)} - \Vert d_\tau \phi_\tau \Vert^2 + (n-k) \Vert \phi_\tau \Vert^2 =  - \Vert d \phi \Vert^2 + (n-k) \Vert \phi \Vert^2$$
where the last equality is by Proposition \ref{resultfrombs} (4).
To sum up, we got that 

$$ \sum_{\gamma \in \Sigma (k+1,\Gamma)} \dfrac{m(\gamma )}{(k+2)! \vert \Gamma_\gamma \vert}  \sum_{\begin{scriptsize} \begin{array}{c} \sigma,\sigma' \in \Sigma (k) \\ \sigma,\sigma' \sqsubset \gamma, \sigma \neq \sigma' \end{array} \end{scriptsize}} cos(X_{\sigma \cap \sigma'} ,\Gamma_{\sigma \cap \sigma'},\rho) \vert \phi (\sigma) \vert \vert \phi (\sigma') \vert \geq $$
$$ \geq - \dfrac{n-k}{n-k+1} \Vert \delta \phi \Vert^2 - \Vert d \phi \Vert^2 + (n-k) \Vert \phi \Vert^2$$

and therefore 
$$ \sum_{\gamma \in \Sigma (k+1,\Gamma)} \dfrac{m(\gamma )}{ (k+2)! \vert \Gamma_\gamma \vert} \sum_{\sigma \in \Sigma (k), \sigma \sqsubset \gamma} \vert \phi (\sigma) \vert^2 - $$
$$ - \sum_{\gamma \in \Sigma (k+1,\Gamma)} \dfrac{m(\gamma )}{(k+2)! \vert \Gamma_\gamma \vert} \sum_{ \begin{scriptsize}
\begin{array}{c} \sigma,\sigma' \in \Sigma (k) \\ \sigma,\sigma' \sqsubset \gamma, \sigma \neq \sigma' \end{array} \end{scriptsize}} cos(X_{\sigma \cap \sigma'} ,\Gamma_{\sigma \cap \sigma'},\rho) \vert \phi (\sigma) \vert \vert \phi (\sigma') \vert \leq$$
$$\leq (n-k) \Vert \phi \Vert^2 + \dfrac{n-k}{n-k+1} \Vert \delta \phi \Vert^2 + \Vert d \phi \Vert^2 - (n-k) \Vert \phi \Vert^2 = $$
$$  \dfrac{n-k}{n-k+1} \Vert \delta \phi \Vert^2 + \Vert d \phi \Vert^2$$
so we got that under the conditions of the theorem
$$  \dfrac{n-k}{n-k+1} \Vert \delta \phi \Vert^2 + \Vert d \phi \Vert^2 \geq  \varepsilon (n-k) \Vert \phi \Vert^2$$
and we are done.

\end{proof}

\begin{corollary}
If there is an $\varepsilon >0$ such that for every $\gamma \in \Sigma (k+1,\Gamma)$ the matrix $A (\gamma, \Gamma)$ is positive definite with eigenvalues greater or equal to $\varepsilon$ then $L^2 H^k(X,\rho)=0$ for every $\rho$. And in the same way, if there is an $\varepsilon >0$ such that for every $\gamma \in \Sigma (k+1,\Gamma)$ the matrix $A (\gamma)$ is positive definite with eigenvalues greater or equal to $\varepsilon$ then $L^2 H^k(X,\rho)=0$ 
\end{corollary}

\begin{corollary}
Define the reduced cosine as follows:
$$ cos_{r} (X,\Gamma,\rho) = sup \lbrace \dfrac{\sum_{(u,v) \in \Sigma (1, \Gamma)} m((u,v)) \langle \phi (u) , \phi (v) \rangle }{\sum_{(u,v) \in \Sigma (1, \Gamma)} m((u,v)) \vert \phi (u)\vert \vert \phi (v) \vert}: 0 \neq \phi \in C^0 (X,\Gamma), \phi = \phi^1  \rbrace$$
and let $A_r (\gamma,\Gamma )$ be the corresponding matrix.  
If for every $\gamma \in \Sigma (k+1,\Gamma)$ the matrix $A_r (\gamma,\Gamma )$ is positive definite then $L^2 \mathcal{H}^k(X,\rho)=0$ for every $\rho$. 
\end{corollary}

\begin{proof}
We just repeat the above proof of the theorem with $0 \neq \phi \in ker (\delta)$ (which means $\phi = \phi^1$) and get $\Vert d \phi \Vert^2 > 0 $. So $ker (\delta) \cap ker (d) = 0$ and we are done. 
\end{proof}

Applying the above corollary to the $2$-dimensional case gives Theorem 2.

\begin{remark}
The results stated above where inspired by \cite{K}, we do not reproduce the results of \cite{K} for dimension $n$ larger than $2$. In \cite{K} the cosine matrix refers always to the angles between the $n-1$ simplices to get property (T), but in our result the cosine matrix of those angles gives only the vanishing of the $n-1$ reduced cohomology and to get property (T) one needs to look at the cosine matrix of angles between $1$ simplices.        
\end{remark}

\section{Examples}
\subsection{GABs}
A generalized $m$-gon (or in another name, a 1-dimensional spherical building) is a connected bipartite graph $L=(V,E)$ of diameter $m$ and girth $2m$ in which each vertex lies on at least two edges. Denote $V=V_1 \cup V_2$ were $V_1,V_2$ are the two sides of the graph (there is an edge between two vertices only if one belong to $V_1$ and the other to $V_2$).  A generalized $m$-gon is said to have parameters $(s,t)$ if every vertex in $V_1$ has valency $s+1$ and every vertex in $V_2$ has valency $t+1$. A generalized $m$-gon is called thick if $s \geq 2,t \geq 2$. A theorem by Feit and Higman \cite{FH} states that a thick $m$-gon exists only if $m=2,3,4,6,8$. Moreover, Feit and Higman computed the smallest positive eigenvalue for the Laplacian on general $m$-gon of type $(s,t)$ and those are given in the list below
\begin{enumerate}
\item For $m=2$ the generalized $m$-gon is a complete bipartite graph and the smallest positive eigenvalue of the Laplacian is always $1$ (is does not depend on $(s,t)$).
\item For $m=3$ one always have $s=t$ and the generalized $m$-gon is the flag complex of a projective plane. In that case the smallest positive eigenvalue of the Laplacian is  
$$1 - \dfrac{\sqrt{s}}{s+1}$$
\item For $m=4$ the smallest positive eigenvalue of the Laplacian is
$$1 - \sqrt{\dfrac{s+t}{(s+1)(t+1)}}$$
\item For $m=6$ the smallest positive eigenvalue of the Laplacian is
$$1 - \sqrt{\dfrac{s+t+\sqrt{st}}{(s+1)(t+1)}}$$
\item For $m=8$ the smallest positive eigenvalue of the Laplacian is
$$1 - \sqrt{\dfrac{s+t+\sqrt{2st}}{(s+1)(t+1)}}$$
\end{enumerate}

"Geometries that are almost buildings" (GABs) and "Chamber system that are almost buildings" (SCABs) were introduced by Tits in \cite{Tits} (other names for those are "geometries of type M" and "chamber systems of type M"). GAB and SCABs can be viewed as finite $n$-dimensional simplicial complexes whose links of codimension $2$ are generalized polygons (exact definition of GABs and SCABs can be found in \cite{LecNot} and in \cite{Tits}). Below we introduce three examples of GABs and SCABs given in \cite{Kantor} and in \cite{LecNot}, whose universal cover are an exotic affine buildings (i.e. affine building that do not arise from a local field) and whose fundamental groups acts on those buildings. These examples are interesting in our context because they fail to meet \.{Z}uk's criterion for property (T), but they meet our criterion given in corollary \ref{T for 2d}. Another interesting thing about the examples below is that two of them are connected to sporadic simple groups. The reader should note that connection of the second example to one of the Fischer's group isn't straightforward - the Fischer's group does not act on the GAB but is related to it . In the table below the first column indicates the Coxeter diagram (with $(s,t)$ written above every link), the second column indicates the group associated with the GAB (if such exists), the third column indicates the universal cover and the last column indicates the reference from which the example was taken.

\begin{center}
\textbf{Table 1 - GAB's which meet out criterion}
\begin{tabular}{|c|c|c|c|}
\hline
diagram & group & universal cover & reference \\
\hline
\includegraphics[scale=0.5]{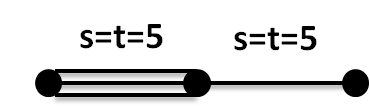} & Lyons & $\tilde{G}_2$ & \cite{Kantor}, 4th example  \\
\hline
\includegraphics[scale=0.5]{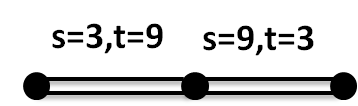} & Fischer & $\tilde{C}_2$ & \cite{Kantor}, 2th example \\
\hline
\includegraphics[scale=0.5]{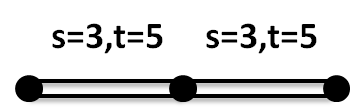} & - & $\tilde{C}_2$ & \cite[C.6.10]{LecNot} \\
\hline

\end{tabular}   
\end{center}

We will show that the first of those examples meets our criterion (checking the other two examples is left to the reader). In the first example, there are three types of links:
\begin{enumerate}
\item Bipartite graph in which $\lambda = 1$ and therefore $\overline{\lambda}_1=\frac{1}{2}$
\item Generalized $3$-gon with $s=5$ in which $\overline{\lambda}_2 = \frac{1}{2} -  \dfrac{\sqrt{5}}{6}$
\item Generalized $6$-gon with $s=t=5$ in which $\overline{\lambda}_3 = \frac{1}{2} -  \dfrac{\sqrt{15}}{6}$
\end{enumerate}
Therefore 
$$\overline{\lambda}_2 + \overline{\lambda}_3 = 1 - \dfrac{\sqrt{5}+\sqrt{15}}{6} < 0$$
so \.{Z}uk's criterion does not hold, but 
$$ \overline{\lambda}_1 + \overline{\lambda}_2 + \overline{\lambda}_3 = \dfrac{3}{2} - \dfrac{\sqrt{5}+\sqrt{15}}{6} >0$$
and 
$$ (\overline{\lambda}_1 + \overline{\lambda}_2)(\overline{\lambda}_1 + \overline{\lambda}_3) +(\overline{\lambda}_1 + \overline{\lambda}_2)(\overline{\lambda}_2 + \overline{\lambda}_3) + (\overline{\lambda}_1 + \overline{\lambda}_3)(\overline{\lambda}_2 + \overline{\lambda}_3) =$$
$$= 3\frac{5}{9}+\dfrac{15 \sqrt{3}}{36} - \frac{4}{6} (\sqrt{5} + \sqrt{15}) > 0$$
therefore the criterion stated in the corollary holds. 

Note that since this example was constructed using the Lyons group which act on the GAB, there is a natural extension of the Lyons group by the fundamental group of the GAB and this extension also has property (T) (either because it acts on the universal cover or becuase it is an extension of a finite group by a group with property (T)). 

\subsection{Hyperbolic buildings}

A Dynkin diagram is of compact hyperbolic type, if it is not of finite or affine type, but every proper subdiagram is finite. For every compact hyperbolic Dynkin diagram and every finite field $\mathbb{F}_q$, Tits  \cite{Tits2} constructed a Kac-Moody group, acting cocompactly on an hyperbolic building with thickness $q+1$. Compact hyperbolic Dynkin diagram were classified in \cite{Car} but we will deal only in the $2$-dimensional case, i.e. the Dynkin diagram has $3$ vertices. Since the links are again generalized polygons, we can use the results of \cite{FH} again for the smallest positive eigenvalue of the Laplacian of each link.
The table below compares the minimal $q$ needed to assure property (T) using our Theorem 1 and \cite[Theorem 1]{Zuk} (cases in which both theorems give the same $q$ were omitted). 
\begin{center}
\textbf{Table 2 - comparison of criteria for $2$ dimensional hyperbolic buildings}
\begin{tabular}{|c|c|c|}
\hline
diagram & minimal $q$ by Theorem 1  & minimal $q$ by \cite[Theorem 1]{Zuk} \\
\hline
\includegraphics[scale=0.25]{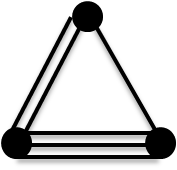} & 7 & 8 \\
\hline
\includegraphics[scale=0.25]{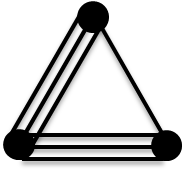} & 9 & 11 \\
\hline
\includegraphics[scale=0.25]{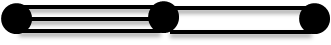} & 5 & 8 \\ 
\hline
\includegraphics[scale=0.25]{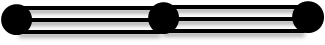} & 7 & 11 \\
\hline
\end{tabular}   
\end{center}
\bibliographystyle{alpha}
\bibliography{bibl}

\end{document}